\documentclass[reqno]{amsart}
\usepackage{latexsym}
\usepackage{xspace}
\usepackage[bookmarksnumbered,colorlinks]{hyperref}
\usepackage{graphics}

\newcommand{\bt}{\begin{teo}}                               
\newcommand{\et}{\end{teo}}                                 

\newcommand{\bco}{\begin{cor}}                               
\newcommand{\eco}{\end{cor}}                                 
\newcommand{\bd}{\begin{defn}}                            
\newcommand{\ed}{\end{defn}}                              
\newcommand{\bl}{\begin{lem}}                                 
\newcommand{\el}{\end{lem}}                                   
\newcommand{\bpr}{\begin{prop}}                  
\newcommand{\epr}{\end{prop}}                    
\newcommand{\bere}{\begin{remark}}                      
\newcommand{\ere}{\end{remark}}                                 
\newcommand{\beq}{\begin{equation}}
\newcommand{\eeq}{\end{equation}}
\def\bal#1\eal{\begin{align}#1\end{align}}                      %
\def\baln#1\ealn{\begin{align*}#1\end{align*}}          
\def\bml#1\eml{\begin{multline}#1\end{multline}}        %
\def\bmln#1\emln{\begin{multline*}#1\end{multline*}}  %
\def\bga#1\ega{\begin{gather}#1\end{gather}}
\def\bgan#1\egan{\begin{gather*}#1\end{gather*}}
\newcommand{\de}{\mathrm{d}}                        
\DeclareMathOperator{\cat}{cat}
\newcommand{\N}{\ensuremath{\mathbb{N}}\xspace}     
\newcommand{\R}{\ensuremath{\mathbb{R}}\xspace}     
\newcommand{\eps}{\varepsilon}                      
\newcommand{\To}{\longrightarrow}                   
\newcommand{\inte}{\int_0^1\!\!}
\newcommand{\gammap}{\ensuremath{\gamma_{\Pi}}\xspace}

\newtheorem{teo}{Theorem}[section]

\newtheorem{lem}[teo]{Lemma}
\newtheorem{defn}[teo]{Definition}
\newtheorem{cor}[teo]{Corollary}
\newtheorem{prop}[teo]{Proposition}
\theoremstyle{remark}
\newtheorem{remark}[teo]{Remark}

\newcommand{\be}{\begin{equation}}
\newcommand{\ee}{\end{equation}}

\newcommand{\gam}{\gamma}
\newcommand{\Om}{\ensuremath{\Omega(p,q;D)}\xspace}

\title[Convex domains of Finsler and Riemannian manifolds]{Convex domains of Finsler and\\Riemannian manifolds}

\author[R. Bartolo]{Rossella Bartolo}
\address{Dipartimento  di Matematica\hfill\break\indent
Politecnico di Bari \hfill\break\indent Via Orabona 4,
70125, Bari, Italy}
\email{rossella@poliba.it}

\author[E. Caponio]{Erasmo Caponio}
\address{Dipartimento di Matematica\hfill\break\indent
Politecnico di Bari \hfill\break\indent Via Orabona 4,
70125, Bari, Italy}
\email{caponio@poliba.it}

\author[A.V. Germinario]{Anna Valeria Germinario}
\address{Dipartimento  di Matematica\hfill\break\indent
Universit\`a degli Studi di Bari\hfill\break\indent Via  Orabona 4,
70125 Bari Italy}
\email{germinar@dm.uniba.it}
\thanks{RB, EC and AVG are supported by M.I.U.R. Research project PRIN07 ``Metodi Variazionali \hfill\break\indent e Topologici nello Studio di Fenomeni Nonlineari''
}

\author[M. S\'anchez]{Miguel S\'anchez}
\address{Departamento de Geometr\'{\i}a y Topolog\'{\i}a \hfill\break\indent
 Facultad de Ciencias, Universidad de Granada \hfill\break\indent
 Campus Fuentenueva s/n, 18071 Granada, Spain}
\email{sanchezm@ugr.es}
\thanks{MS is partially supported by
Spanish MEC-FEDER Grant MTM2007-60731 and
Regional J.\hfill\break\indent Andaluc\'{\i}a Grant P09-FQM-4496.}

\thanks{All the authors are partially supported by the Spanish-Italian \hfill\break\indent  Acci\'on Integrada HI2008.0106/Azione Integrata Italia-Spagna IT09L719F1.}

\subjclass[2000]{53C60, 53C22, 58E10}

\keywords{Finsler manifolds, convex hypersurface, 
geodesics}

\date{}

\begin{document}

\begin{abstract}
A detailed study of the  notions of convexity for a hypersurface  in a Finsler manifold  is
carried out. In particular, the infinitesimal and local notions of
convexity are shown to be equivalent.
Our approach differs from
Bishop's one in his classical result \cite{b} for the Riemannian case. Ours not only can be extended to the Finsler setting but
it also reduces the typical requirements of differentiability for the
metric and it yields consequences on the multiplicity of connecting
geodesics in the convex domain defined by the hypersurface.
\end{abstract}

\maketitle
\section{Introduction}\label{s1}
Convexity is a central concept in different branches of
Mathematics and, thus, it admits different definitions depending
on the used viewpoint. In Riemannian Geometry, there are two
natural definitions for the convexity of a smooth hypersurface
$\partial D$ which bounds a {\em domain} $D$, i.e. a connected open subset: $\partial D$ is
{\em infinitesimally convex} if its second fundamental form, with
respect to the inner normal, is positive semi-definite at any
$p\in\partial D$ and {\em locally convex} if the exponential of the
tangent space $T_p\partial D$, restricted to some neighborhood of
$0$, does not intersect $D$. Infinitesimally convex hypersurfaces
naturally arise from regular values of smooth convex functions.
On the other hand, the domain $D$ is called {\em convex} when each
two $x, y \in D$ can be joined by a non-necessarily unique geodesic which minimizes the
distance in $D$. When the closure $\overline D$ is complete, the
convexity of $D$ must be equivalent to the convexity of its
boundary; in order to prove this claim, an intermediate notion
such as {\em geometric convexity} for $\partial D$  becomes useful
(see the next section for exhaustive definitions and details).

The consistency of the approach relies then on the equivalence
between the infinitesimal and local  notions of convexity. The
fact that the former implies the latter is not as trivial as it
sounds: do Carmo and Warner \cite{dw} prove it when the ambient
Riemannian manifold $(M,g_R)$ has constant curvature and Bishop
\cite{b} in the general situation. Bishop's proof reduces the
problem to dimension 2 and, to this end, a family of surfaces
which sweep out a neighborhood of $p$ is constructed. As a
consequence, a uniform bound for some focal distances in the
family is required, and smoothability $C^4$ is imposed on the
metric $g_R$. This requirement on smoothability seems  strong and
Bishop himself suggests that it may be non-optimal.  Notice also that the
interplay between bounds on curvature, convexity of domains or
functions and smoothability,  becomes a classical topic in
Differential Geometry  \cite{Ba1, Ba2, ChGr, GrWu, Sm}.

Most elements of this Riemannian setting can be transplanted to
the Finslerian one. But, as pointed out by Borisenko and Olin
\cite[Remark 1]{bo}, an important difficulty now appears: Bishop's
technique only works for Berwald spaces, where the Chern connection becomes a linear connection on the tangent bundle (cf. \cite[\S 10]{bcs}). 
In the general case, the relation between
the convexity of the domain and its boundary is not clear and one
is lead to some more strict notions of convexity as a technical
assumption (cf. for example \cite{Ba1, Sm}).
The situation is even worse for
non-reversible Finsler metrics, as there is no a priori a clear
equivalent hypothesis to the completeness of $\overline D$.

The aim of the present article is to give a definitive answer to
these questions, showing the natural equivalence of the different
convexities.

As a preliminary step, in Section \ref{s15} the different notions of
convexity are reviewed, explaining their extensions to the
Finslerian case and checking that, in the non-reversible case, it
is equivalent to assume any notion of convexity for the original
Finsler metric $F$ and its reversed metric
$\tilde F$.

In Section \ref{s2}, the equivalence between infinitesimal and
local convexities is proved. Indeed the following result holds:

\begin{teo}{\bf [Finslerian Bishop's Theorem]}\label{mainmain} Let $M$ be a smooth manifold, 
endowed with a  Finsler metric 
whose fundamental tensor  (\ref{fundmatr}) is $C^{1,1}_{\rm loc}$ 
(i.e. its components are $C^{1}$ in $TM\setminus 0$ with locally Lipschitz  derivatives)
and let $N\subset
(M,F)$ be a $C_{\rm loc}^{2,1}$ embedded hypersurface (i.e., $N$
is locally  regarded as the inverse image of some $C^{2,1}$ regular function).

Let $p\in N$ and choose a transverse direction as inner pointing
in some neighborhood $U$ of $p$. If $N$ is infinitesimally convex
in $U\cap N$, then $N$ is locally convex at $p$ (and, thus, on all $U\cap N$).
\et
Notice that this extension of Bishop's theorem to the Finsler case
is also useful in the Riemannian setting, as in this case only $C^{1,1}$
differentiability is required for the metric. In regard to the hypersurface $N$,
it is not clear if more specific techniques on regularity may reduce
$C_{\rm loc}^{2,1}$ in $C^2$.
However, as pointed out by Li and Nirenberg \cite{LN}, the
hypothesis $C_{\rm loc}^{2,1}$ is the natural regularity
assumption when the distance function to a boundary is considered
(see also \cite[Sect. 4, 5]{cys}).

Theorem \ref{mainmain}, combined with the
straightforward implications discussed in Section \ref{s15}, yields
the full equivalences among the notions of convexity for the
boundary of a domain, namely:

\begin{cor}\label{coor}
Let $D$ be a
$C^{2,1}_{\rm loc}$ domain (i.e. an open connected subset of $M$ whose boundary is locally defined as a level set of a $C^{2,1}$ function) in a manifold endowed with a  Finsler
metric  whose fundamental tensor is  $C^{1,1}_{\rm loc}$ on $TM\setminus 0$. It is equivalent for $\partial D$ to be: (a)
infinitesimally convex, (b) geometrically convex and (c) locally
convex. \end{cor}

For a domain $D$ the correspondence of the equivalent notions of
convexity for $\partial D$ and the convexity of $D$ is summarized
in the following result  proved in Section \ref{s3},  which involves
the balls for the symmetrized distance ${\rm d_s}$
of the pseudo-distance associated to $F$:

\begin{teo}\label{main}
Let $D$ be a $C^{2,1}_{\rm loc}$ domain of a smooth manifold $M$
endowed with a  Finsler metric $F$  having $C^{1,1}_{\rm loc}$ fundamental tensor and such that the intersection
of the closed symmetrized balls $\overline B_s(p,r)$ with
$\overline D$ is compact.

Then, $D$ is convex if and only if $\partial D$ is convex (in any one of the equivalent sense in Corollary~\ref{coor}).
Moreover, in this case,  any pair of points in $D$ can be joined
by infinitely many connecting geodesics contained  in $D$ and having
diverging lengths if   $D$ is not contractible.
\end{teo}
Some remarks are in order. First, the compactness of the
intersections $\overline D \cap \overline B_s(p,r)$,
plays a role analogous
to that of the completeness of $\overline D$ in the Riemannian setting.
This becomes natural after \cite{cys}, where a
correspondence between some elements in Lorentzian and Finslerian
geometries is exploited. For example, the compactness of the
symmetrized balls $\overline B_s(p,r)$ (which is a condition
weaker than forward or backward completeness of $F$) yields the existence of
a minimizing  geodesic between any two points in a Finsler manifold.
Our approach to the problem of the convexity of a domain uses variational methods which directly yield  multiplicity results by standard arguments.
The proof  is  based on a
{\em penalization technique} which goes back to Gordon
\cite{g1}.
Roughly speaking, the lack of completeness impedes the energy functional
(see (\ref{fe}) below) to satisfy the Palais-Smale condition, so that many classical results in critical
point theory are not applicable. Thus, the functional is modified by adding a
penalizing term which becomes infinite close to the boundary and
a family of penalized functionals $(J_\eps)_{\eps>0}$ is
considered. 

However, there are  interesting differences
with respect to the Riemannian setting studied in \cite{g1}. 
In fact, the critical points of
such functionals (which are  approximating solutions) are  $C^1$ curves having supports in $D$. They are continuously twice differentiable only on the open subset of the domain of parametrization where their velocity vector field is not zero. Anyway, due to the particular conservation law they satisfy \eqref{ene}, the 
 set where their derivative vanishes is not negligible and
the passage to the limit ($\eps\to 0$) in the penalization technique is much more delicate than in the Riemannian case. Another difficulty is connected with the fact that 
a result analogous to the Nash
isometric embedding theorem (that in several papers about this topic, allows one to avoid many technicalities,
see for example \cite{g1, bfg, gm, g})  
does not  hold in general for a Finsler manifold, \cite[Theorem 1.1]{Shen98}.

Theorem~\ref{main} extends to domains of Finsler
manifolds a result  about 
Finsler metrics in $\R^N$ (see  \cite[Theorem 1, p. 250]{gh}), stating  the
existence of a geodesic of length equal to the Finslerian distance between any two points in $\R^N$.   
Such a result  is an emblematic example
of the application of the {\em direct method in the calculus of variations} (cf. \cite[Preface]{gh}).  
Namely it comes from a minimization argument, based on the
lower-semicontinuity with respect to the $C^0$-topology of the
energy integral of a Finsler metric   (see \cite[Lemma 5, p. 259]{gh}). The role of the assumption about the compactness of the
sets $\overline B_s(p,r)\cap \overline D$ in Theorem~\ref{main} is
taken in \cite{gh} by the existence of a constant $a>1$ such that
\[\frac{1}{a} |y|\leq F(x,y)\leq a|y| \hbox { for any } x\in \R^N, y\in\R^N.\]
The result in \cite{gh} can be extended to a domain $D$ in $\R^N$ but, as the authors themselves  observe \cite[Remark 3, p. 254]{gh}, it is not easy to give and to check convexity conditions on $\partial D$  ensuring that a minimizing geodesic in $D$ exists. 
Our result on the equivalence between different notions of convexity for $\partial D$  and Theorem~\ref{main} aim also to fill that gap between Riemannian and Finsler geometry.

\section{Preliminaries} \label{s15}
Hopf-Rinow theorem states that metric (or equivalently geodesic)
completeness  of a Riemannian manifold $(M,g_R)$ is a sufficient
condition for convexity.

When  a domain $D$ of a
Riemannian manifold $M$ is  considered, suitable convexity
assumptions are needed in order to control the lack of
completeness. As pointed out by Gordon \cite{g1}, on Riemannian
manifolds this problem is interesting also because of its
relation, via the Jacobi metric, to the problem of connecting two
points by means of a trajectory of fixed energy for a Lagrangian
system. In the case of a Finsler manifold the study of convexity
of a domain  is interesting also in connection with the
existence of lightlike and timelike geodesics connecting a point
with a line in an open region of a  stationary spacetime (see \cite{cym}).

At first we review the different notions of convexity for the
boundary  of a smooth domain $D$ of a Riemannian manifold  (see
\cite{bgs} and \cite{s} where also non differentiable boundaries
are considered).

We say that the boundary $\partial D$ of $D$ at a point $x\in
\partial D$ is
\begin{itemize}
\item {\em infinitesimally convex} if the second fundamental form
$\sigma_x$, with respect to the interior normal, is positive
semi-definite; \item {\em locally convex} if there exists a
neighborhood $U \subset M$ of $x$ such that \be\label{espo}
\exp_x\left( T_x\partial D\right)\cap\left(U\cap D\right) =
\emptyset. \ee
\end{itemize}
In order to apply variational methods to the study of geodesic
connectedness, a characterization of the infinitesimal convexity
is useful. Indeed, note that  for each $x\in\partial D$ a neighborhood $U\subset M$ of $x$ and a differentiable
function (with the same degree of differentiability of $\partial
D$) $\phi: U \rightarrow\R$ exist such that \be\label{defb}
\begin{cases}
          \phi^{-1}(0) =  U\cap\partial D \\
          \phi > 0&\mbox{on $U\cap D$ }\\
          {\rm d}\phi(x)\not= 0 & \text{for every $x\in U\cap \partial D$}
 \end{cases}
\ee and then the following holds:
\begin{itemize}
\item $\partial D$ is infinitesimally convex  at $x\in\partial D$
if and only if for one (and then for all) function $\phi$
satisfying (\ref{defb}):
$$H_\phi(x)[y,y]\leq 0 \quad \text{for every $y\in T_x\partial D$.}
$$
\end{itemize}
Easily, the local convexity at $x$ implies the infinitesimal one. For the converse, one has to
assume that the infinitesimal convexity holds on a neighborhood of $x$; in this case, Bishop \cite{b}
proved that the converse holds if the metric is $C^4$.
 Notice that $\partial D$ is assumed to be an embedded manifold in $M$, so  that the function $\phi$ in
(\ref{defb}) can be found as defined on all $M$. Thus, the following
global definitions (equivalent at least in the $C^4$ case) can be given:
\begin{itemize}
\item $\partial D$ is {\em infinitesimally convex} if so it is at
any point;
\item $\partial
D$ is {\em variationally convex} if for one, and then for all,
function $\phi$ on $M$ such that \be
\label{defbt}
\begin{cases}
            \phi^{-1}(0) = \partial D \\
            \phi > 0 & \mbox{on $D$} \\
            \de\phi(x)\not= 0& \text{for every $x\in\partial D$}
\end{cases}
\ee it holds
\[
H_\phi(x)[y,y]\leq 0 \quad \text{for every $x\in\partial D$, $y\in
T_x\partial D$.}
\]
\item $\partial D$ is {\em locally
convex} if so it is at any point.
\end{itemize}
There is also another  definition of convexity which comes out to be equivalent to the previous ones.
\begin{itemize}
\item $\partial D$ is {\em geometrically convex} if for any
$p,q\in D$ the range of any geodesic
$\gamma:[a,b]\rightarrow\overline D$ such that $\gamma(a) = p$ and
$\gamma(b) = q$  satisfies
\[
\gamma\left([a,b]\right)\subset D.
\]
\end{itemize}
This definition is intermediate between infinitesimal/variational convexity and local convexity.
In fact, under geometric convexity, any geodesic $\rho: ]-\epsilon,\epsilon[\rightarrow M$ with initial velocity in $T_x\partial D$ will remain
in $M\setminus D$ reducing eventually  $\epsilon >0$. This implies infinitesimal convexity at $x$ but, in order to obtain local convexity, one must ensure that the same $\epsilon$ can be chosen in all the directions.  Bishop's result ensures the equivalence in the $C^4$ case (see also \cite{g} for a different technique in one of the implications). As we will see in Proposition~\ref{pA} and Corollary~\ref{cor}, the equivalences hold even for $C^{1,1}$ Riemannian metrics, as these equivalences hold in the general Finslerian case whenever the fundamental  tensor has such a level of regularity.

There are also different ways to prove that, for a complete $M$
(or equivalently $\overline D$), the boundary $\partial D$ is
convex if and only if the domain $D$ is convex (see the review \cite{s}).

We deal here with convexity of a domain $D$ of a
Finsler manifold $M$, so let us
recall some basic notions in Finsler Geometry.

A Finsler structure  on a smooth finite dimensional manifold $M$
is a function $F\colon TM\to[0,+\infty)$ which is continuous on
$TM$, smooth on $TM\setminus 0$, vanishing only on the zero
section, fiberwise positively homogeneous of degree one, i.e.
$F(x,\lambda y)=\lambda F(x,y)$, for all $x\in M$,   $y\in T_x M$
and $\lambda>0$, and which has fiberwise strictly convex square
i.e. the matrix \beq\label{fundmatr}
g(x,y)=\left[\frac{1}{2}\frac{\partial^2 (F^2)}{\partial
y^i\partial y^j}(x,y)\right]\eeq is positively defined for any
$(x,y)\in TM\setminus 0$.  Typically, the word ``smooth'' means $C^{\infty}$  and one can maintain this here for the manifold $M$. Nevertheless,
in order to obtain a sharp result on differentiability, $F$ {\em smooth} will mean that the fundamental
tensor $g$ in \eqref{fundmatr} is $C^{1,1}_{\rm loc}$. Obviously, this will hold for the Finsler metric associated to a $C^{1,1}$  Riemannian metric as well as when $F$ is a $C^{3,1}_{\rm loc}$ function on $TM\setminus 0$.
The length of a piecewise smooth  curve
$\gamma\colon [a,b]\to M$ with respect to the Finsler structure
$F$  is defined by
\[\ell_F(\gamma)=\int_a^b\!\!
F(\gamma,\dot\gamma)\; \de s\] hence the distance  between two
arbitrary points $p, q\in M$ is given by
\[
\de (p,q)= \inf_{\gamma\in {\mathcal P}(p,q; M)}{\ell}_F(\gamma),
\] where ${\mathcal P}(p,q;M)$ is the set of all piecewise smooth
curves $\gamma\colon[a,b]\to M$ with $\gamma(a)=p$ and
$\gamma(b)=q$.  The distance function is non-negative and
satisfies the triangle inequality, but it is not symmetric since
$F$ is only positively homogeneous of degree one in $y$. So for
any point $p\in M$ and for all $r>0$ we can define two different
balls centered at $p$ and having radius $r$: the {\em forward
ball} $B^+(p,r)=\{q\in M \mid \de (p,q)<r\}$ and the {\em
backward} one $B^-(p,r)=\{q\in M \mid \de (q,p)<r\}$. Analogously,
it makes sense to give two different notions of Cauchy sequences
and completeness: a sequence $(x_n)_n\subset M$ is a {\em forward}
(resp. {\em backward}) {\em Cauchy sequence} if for all $\eps>0$
there exists an index $\nu\in\N$ such that, for all $m\geq n\geq
\nu$, it is  $\de(x_n,x_m)<\eps$ (resp. $\de (x_m,x_n)<\eps$);
consistently a Finsler manifold is {\em forward complete} (resp.
{\em backward complete}) if every forward (resp. backward) Cauchy
sequence converges. It is well known that both the topology
induced by the forward  balls and that induced by the backward
ones agree with the underlying manifold topology. Moreover
suitable versions of the Hopf-Rinow theorem hold (cf.
\cite[Theorem 6.6.1]{bcs}) stating, in particular, the equivalence
of    forward (resp. backward) completeness and the compactness of
closed and forward (resp. backward) bounded subsets of $M$. The
validity of one of these properties implies the existence of a
geodesic connecting any  pair of points in $M$ and minimizing the
Finslerian distance, i.e. the convexity of $M$. Geodesics can be defined in different ways using different
connections defined on the pulled-back bundle $\pi^*TM$,
$\pi\colon TM\to M$, (cf. \cite[Chapter 2]{bcs}) or as critical
points of the length functional (cf. \cite[Proposition 5.1.1]{bcs}
for details); furthermore (cf. for example \cite[Proposition 2.3]{cym}) a
smooth curve $\gamma$ on $[a,b]$ is a
geodesic parameterized with constant speed (i.e $s\mapsto F(\gamma(s),\dot\gamma(s))={\rm const.}$) if and only if it is a critical point of the energy functional
\be\label{fe} J(\gamma)=\frac 12\int_a^b F^2(\gamma,\dot \gamma)\;
\de s \ee 
defined on the manifold of $H^1$ curves having fixed
endpoints. Thus in local coordinates $\gam$ 
satisfies the equations 
\beq\label{spray}
\ddot\gam^i(s) +g^{ij}(\gamma,\dot\gamma)\Big(\frac 12\partial_{y^jx^k}F^2(\gamma,\dot\gamma)\dot\gamma^k-\frac12\partial_{x^j}F^2(\gamma,\dot\gamma)\Big)=0.
\eeq
Here $g^{ij}$
are the components of the inverse matrix of fundamental tensor $g$, $\partial_{x^j}$, $\partial_{y^jx^k}$ are the symbols of the partial derivatives with respect to the variables $x^j, y^j$ and we adopt  the usual Enstein's summation convention. Using the structural equations defining the Chern connection (see  \cite[Theorem 5.2.2]{s2}) it can be proved that the functions 
$G^i(x,y)=g^{ij}(x,y)\big(\tfrac 12\partial_{y^jx^k}F^2(x,y)y^k-\tfrac12\partial_{x^j}F^2(x,y)\big)$, $(x,y)\in TM\setminus 0$, are equal to
$\Gamma^i_{jk}(x,y)y^jy^k$, where $\Gamma^i_{jk}$ are the components of the Chern
connection.\footnote{To see that use Eqs. (5.2), (5.7) and the formula after equation (5.31) in \cite{s2}, besides the fact that the functions $G^i$ are positively homogeneous of degree $2$ in the $y$ variable.} Therefore geodesics equations become
\beq\label{eqgeo}
 \ddot\gam^i(s) +\Gamma^i_{jk}(\gamma(s),\dot\gamma(s))\dot\gamma^j(s)\dot\gamma^k(s)=0.
\eeq 
\bere\label{sym}
Consider the {\em
symmetrized distance} on $M$
\[
\de_s(p,q)=\frac 12 \left(\de(p,q) + \de(q,p)\right).
\]
and denote by $B_s$ the balls associated to $\de_s$. It results that
if the Heine-Borel property holds, i.e. for all $x\in M, r>0$, the
closed balls $\overline B_s(x,r)$ are compact (or equivalently the
subsets  $\overline B^+(x,r_1)\cap \overline B^-(y,r_2)$ are
compact for any $x,y\in M, r_1, r_2>0$), then the metric space
$(M,\de_s)$ is complete (cf. \cite[Proposition 2.2]{cys}).

This condition
implies convexity (cf. \cite[Theorem 5.2]{cys}). It is worth to
stress that the Hopf-Rinow theorem in general does not hold for
the metric $\de_s$. For instance, Example  2.3 in \cite{cys}
exhibits  a non compact, $\de_s$-bounded Randers space whose
symmetrized distance $\de_s$ is complete.
\ere

Now, let $N$ be a hypersurface of $M$ and choose a (unit)
normal vector $n$  at some $x\in N$ (namely, the hyperplane
parallel to $T_xN$ through $n$ is tangent to the $F$-unit ball at
$T_xM$). The {\em normal curvature} $ \Lambda_n$ at a point $x\in
N$ in a direction $y\in T_x N$ is defined by
\beq\label{normcurv}
\Lambda_n(y)=
g(x,n)[\nabla_{\dot\gamma(s)}\dot\gamma(s)\mid_{s=0},  n]
\eeq
where $\gamma\colon ]-\eps,\eps[\to N$ is a geodesic for the
Finsler metric induced by $F$ on  $N$ such that $\gam(0)=x,\
\dot\gam(0)=y$, $\nabla_{\dot\gam}\dot\gam$ is the {\em covariant
derivative} of $\dot\gam$ along $\gam$ in $(M,F)$ (see \cite[\S
5.3]{s2}). Observe that the  definition  of the normal curvature
$\Lambda_n(y)$ in \eqref{normcurv} differs from the one in \cite[\S 14.2]{s2} for a minus sign.

When we deal with domains $D$ of Finsler manifolds, we say (cf.
\cite[Proposition 14.2.1]{s2}) that the boundary $\partial D$ of
$D$ at a point $x\in \partial D$ is
\begin{itemize}
\item {\em infinitesimally convex} if the normal curvature  with respect to the normal vector 
pointing into $D$ is
non-negative or equivalently if  for a function $\phi$ as in
(\ref{defb}): 
\be\label{fc} H_\phi(x,y)[y,y]\leq 0 \quad \text{for
every $y\in T_x\partial D$}, \ee
\end{itemize}
where $H_\phi$ is the Finslerian Hessian of $\phi$ (see \cite[\S
14.1]{s2}) defined, for each $(x,y)\in TM\setminus 0$, as
$H_\phi(x,y)[y,y]=\frac{\de^2 }{\de s^2}(\phi\circ \gam)(0)$, 
being $\gamma$  the geodesic of $(M,F)$ (parameterized with
constant speed) such that $\gam(0)=x$ and $\dot\gam(0)=y$. Taking
into account the equation (\ref{eqgeo}) satisfied by constant
speed geodesics, we get  in local coordinates 
\beq\label{hessF}
(H_\phi)_{ij}(x,y)y^iy^j= \frac{\partial^2\phi}{\partial
x^i\partial x^j}(x)y^iy^j - \frac{\partial\phi}{\partial
x^k}(x)\Gamma^k_{ij}(x,y)y^iy^j, 
\eeq 
being $\Gamma^k_{ij}(x,y)$
the components of the Chern connection of $(M,F)$.

In general, Finsler metrics are {\em non-reversible}, which means
$F(x,-y)\not=F(x,y)$ on $TM$, thus we can define the {\em
reversed} Finsler metric $\tilde F$ as $\tilde F(x,y)=F(x,-y)$ for
each $(x,y)\in TM$.

If $F$  is non-reversible and $\gam$ is a geodesic on $[0,1]$, the
reversed curve $\tilde \gam(s)=\gam (1-s)$ in general is not a
geodesic of $F$, but it is a geodesic for $\tilde F$  (this
can be easily seen by using the fact that geodesics are the curves
that locally minimize the length functional and that $\ell_{\tilde
F}(\tilde \gam)=\ell_F(\gam)$).

\bere\label{infconvequiv}
The notions of infinitesimal convexity for $F$ and $\tilde F$
are equivalent: indeed, if $x\in D, y\in T_x\partial D$ and
$\tilde\gamma$ is the geodesic for $(M,\tilde F)$ such that
$\tilde\gam(0)=x, \dot{\tilde\gam}(0)=-y$, then
$\gam(s)=\tilde\gam(-s)$ is a geodesic in $(M,F)$ (and $\gam(0)=x,
\dot\gam(0)= y$); since the components $\tilde\Gamma^i_{jk}$ of
the Chern connection of $\tilde F$ satisfy
$\tilde\Gamma^i_{jk}(x,-y)=\Gamma^i_{jk}(x,y)$, we have
$H_\phi^{\tilde F} (x,-y)[-y,-y]=H_\phi(x,y)[y,y]\leq 0$.
Moreover, associated to a non-reversible metric, there are two
exponential maps: one, denoted by $\exp$, associated to the
geodesics of $F$,  the other one, denoted by $\widetilde \exp$
associated to $\tilde F$, cf. \cite[Chapter 6]{bcs}). The
definition of {\em local convexity} in \cite[p. 216]{s2} is indeed
equivalent to require that (\ref{espo}) holds for both the
exponential maps.  Recall also that the definition of {\em
geometric convexity} for the Riemannian setting can be extended
trivially to the Finslerian one, and becomes equivalent for geodesics of $F$
and $\tilde F$.\ere

\section{Convexity of the Boundary}\label{s2}
As in the Riemannian case,   local convexity at one point of an hypersurface in a Finsler manifold implies
infinitesimal convexity (cf. \cite[Theorem 14.2.3]{s2}).

We pointed out that, except for Berwald spaces, Bishop's theorem (which ensures that the converse is true in a Riemannian manifold, when infinitesimal convexity holds in a neighborhood of the point in the hypersurface) seems to be an open issue in the Finslerian setting.
We see here that in fact  Bishop's theorem is true also for any Finsler manifold.

Throughout this section $D$ denotes a
$C^{2,1}_{\rm loc}$
domain, that is, there  exists a $C^{2,1}_{\rm loc}$ function $\phi$ which satisfies  (\ref{defbt}). Notice that then $\partial D$ is endowed with an intrinsic $C^{1,1}_{\rm loc}$ structure.

We start by giving two results which generalize   the analogous ones in \cite{g} where geometric convexity of a $C^3$  domain of a complete Riemannian manifold is  studied.
The differential inequality in the next lemma is less restrictive than the one in \cite[Lemma 9]{g}.
\begin{lem} \label{grnew}
If  $\psi \in C^2 ( [0, b] , \R )$ is a non-negative function verifying
\beq  \label{avg}
\begin{cases}
\ddot \psi  \leq A ( \psi + | \dot \psi | ) \\
\psi (0) = 0,  \dot \psi (0) = 0
\end{cases}
\eeq
for some $A >0$, then $\psi \equiv     0$ on $[0,b]$.
\end{lem}
\begin{proof}
By contradiction, assume that a non-trivial solution of \eqref{avg} exists. If $\dot\psi\geq 0$ on the interval $[0,b]$, then  integrating  on  $[0,t]$, $0<t\leq b$, both hand sides in \eqref{avg} we get
\[\dot\psi(t)\leq A\left(\int_0^t\psi(s)\, \de s+\psi(t)\right).\]
Integrating  again we obtain
\[\psi(t)\leq (Ab+1)\int_0^t\psi(s)\, \de s\]
and
 from  Gronwall's inequality we get that $\psi\equiv 0$ in $[0,b]$.
Hence we can assume that a point $\bar t\in [0,b[$, say $\bar t=0$, exists such that $\dot\psi$ has indefinite sign in a right neighborhood of $0$. So a sequence $(t_m)_m$ converging to zero exists such that each $t_m$ is a maximum point of $\psi$:
$$ \psi (t_m )\to 0, \quad \dot \psi (t_m )= 0, \quad \ddot \psi (t_m ) \leq 0.
$$

Now,   let $\psi_m : [ 0 , b] \rightarrow \R$ be the unique solution of
\beq  \label{evg}
\begin{cases}
\ddot \phi  = A ( \phi +   \dot \phi  ) \\
\phi (t_m ) = \psi (t_m) ,  \dot \phi (t_m ) = 0.
\end{cases}
\eeq
We are going to  prove that
\beq \label{evgin}
\psi (t) <\psi_m (t), \quad  \text{for all $t \in]t_m ,b]$.}
\eeq
Notice that easy computations give
\[
\psi_m (t) =  C^-_m e^{\lambda_- t} + C^+_m e^{\lambda_+ t}
\]
where   $\lambda_- < 0 <\lambda_+$  are the two roots
of $\lambda^2 - A \lambda - A =0$ and $C^-_m, C^+_m$ are strictly  positive constants
obtained imposing the initial conditions in \eqref{evg}. Thus $\ddot \psi_m (t) >0$ for any $t$ and, as $\dot\psi_m(t_m)=0$, $\dot \psi_m$ is strictly increasing and positive on $]t_m,b]$.

Since $\psi_m (t_m ) = \psi  (t_m)$, $\dot \psi_m (t_m ) = 0= \dot \psi (t_m)$, $\ddot \psi_m
 (t_m ) > 0 \geq \ddot \psi(t_m)$, inequality \eqref{evgin} is true in
a right neighborhood of $t_m$. To prove that it holds on the whole interval $]t_m ,b]$, assume by contradiction that there exists a point $\tilde t\in ]t_m,b]$ such that $\psi_m(\tilde t)=\psi(\tilde t)$. Let $c>t_m$ be the minimum of the set
\[A=\{t\in]t_m,b]\ |\ \psi(t)=\psi_m(t)\}.\]
Hence $\psi  (t ) <  \psi_m (t )$ for $t\in ]t_m,c[$ and
$\dot \psi_m (c ) \leq  \dot \psi  (c)$. As $0 =\dot  \psi  (t_m ) <  \dot \psi_m
(c) \leq \dot \psi (c)$, we can consider $c_0 = \max \, \{ t \in [t_m , c[\  \mid\  \dot \psi (t) =0 \}$.
It is $\dot \psi >0$ in $]c_0 , c]$ and $c_1 \in ] c_0 , c] $ exists such that $\dot \psi (c_1 ) = \dot \psi_m (c_1) $ and
$\dot \psi (t) < \dot \psi_m (t) $ if $t \in ]c_0 , c_1 [$.
Thus, for any $t \in ]c_0 , c_1[ $, the following inequalities holds
$$\ddot \psi (t) \leq A ( \psi (t) + \dot \psi (t) ) < A  ( \psi_m  (t) + \dot \psi_m (t) ) =
\ddot \psi_m (t)$$
and we get a contradiction observing that
$$\dot \psi (c_1) = \int_{c_0}^{c_1} \ddot \psi (t)\, \de t < \int_{c_0}^{c_1} \ddot \psi_m  (t) \, \de t
< \dot \psi_m (c_1) = \dot \psi (c_1). $$

Inequality \eqref{evgin} allows us to complete the proof. Indeed, as $\psi (t_m )\to 0$, by smooth dependence of the solutions of \eqref{evg} by initial conditions,   the sequence $(\psi_m (t))_m$ goes to $0$, for each $t\in ]0,b]$.
\end{proof}
The following crucial proposition holds. This will turn out to be a
strengthening  of geometric convexity, as it forbids the
possibility of tangency to the boundary for  geodesics  in
$\overline D$ but not lying in $\partial D$.
\begin{prop} \label{pA}
Assume that $\partial D$ is infinitesimally convex in a neighborhood  $U$ of $p\in\partial D$. Let $\gamma: [0,b]\rightarrow U$ be a geodesic which satisfies $\gamma(0)=p, \gamma(]0,b])\subset U\cap D$.
Then, $\dot \gamma(0) \not\in T_p(\partial D)$.
\end{prop}
\begin{proof}
Assume by contradiction that
$\dot\gamma(0)\in T_p\partial D$.
 We are going to prove that $\sigma
>0$ exists such that $\gamma([0,\sigma[)\subset \partial D$, getting a contradiction.

Without loss of generality, we also assume that $\gamma$
is parameterized with unit
Finslerian speed, i.e. $F(\gamma(s),\dot\gamma(s))=1$ for all
$s\in[0,b]$.
Take a chart $(V,(x^i)_{i=1,\ldots,n})$ of $M$
centered  at $p$, with $V\subset U$ and adapted to $D$ (i.e. function  $\phi$ defining the boundary of $D$ is given as a  coordinate $x^i$, say $\phi=x^n$).
In what follows, $|\cdot|$ denotes the
Euclidean norm on $\varphi(V)\subset\R^n$, $\nabla^0$
its
associated gradient and, with an abuse of notation, the symbols
which denote elements in $M$ or $TM$ remain unchanged for the
induced ones by means of $\varphi$ in $\R^n$ or $\R^{2n}$. In
particular,
$\nabla^0\phi \equiv \partial_{x^n}$.
Assuming also that the closure $\overline V$ is compact,
$a>0$ exists such that:
 \beq\label{controllo1} \frac{1}{a}|y|\leq
F(x,y)\leq a|y|, \quad \text{for every $(x,y)\in TV.$}  \eeq Now,
define the natural projection map on $x^n=0$. 
More precisely, let $\eta$ be a local flow around $p$, i.e. for
some $\epsilon>0$, $W=\varphi^{-1}(]-\epsilon,\epsilon[^n)\subset
V$:

$$\eta: ]-\epsilon,\epsilon[ \times W \longrightarrow V\subset M, \quad \quad \eta(t,(x^1,\dots,x^n))=(x^1,\dots,x^{n-1}, x^n-t).$$
Obviously, 
$\phi(\eta(\phi(w),w)) = 0, $ and the projection $\Pi:W
\longrightarrow \partial D$ is defined as:
$$
\Pi(w) = \eta(\phi(w),w).
$$
As $\gamma(0) \in \Pi(W)$, $\sigma  >  0$ exists  such
that $\gamma(s) \in W$ for all
 $s \in
[-\sigma,\sigma]$. Consider the projected curve $\gamma_{\Pi}:
[-\sigma,\sigma] \rightarrow \partial D$ of $\gamma$ on $\partial D$ given by
$\gamma_{\Pi}(s) = \Pi(\gamma(s))$. Since $\dot \gamma_{\Pi}(s)
\in T_{\gamma_{\Pi}(s)}\partial D$, by (\ref{fc}) we have \be\label{dis}
H_{\phi}(\gamma_{\Pi}(s),
\dot\gamma_{\Pi}(s))[\dot\gamma_{\Pi}(s),\dot\gamma_{\Pi}(s)] \leq
0, \quad \text{for every $s\in [-\sigma,\sigma]$.}\ee
Set
$\rho(s)=\phi(\gamma(s))$, it follows $\ddot \rho(s) =
H_\phi(\gamma(s),\dot\gamma(s))[\dot\gamma(s),\dot\gamma(s)]$.
Moreover
\beq\label{dotgammapi}
\dot\gamma_{\Pi}(s)=\de\Pi(\gamma(s))[\dot\gamma(s)]
=\partial_t\eta(\rho(s),\gamma(s))\dot\rho(s)+\partial_x\eta(\rho(s),\gamma(s))[\dot\gamma(s)].\eeq

Using the local expression of the Hessian
of $\phi$ (see (\ref{hessF})),  from (\ref{dis}),  we get on
$[0,\sigma]$: 
\bal\label{hij}
\ddot\rho(s)&=(H_\phi)_{ij}(\gamma(s),\dot\gamma(s))\dot\gamma^i(s)\dot\gamma^j(s)\nonumber\\
&\leq(H_\phi)_{ij}(\gamma(s),\dot\gamma(s))\dot\gamma^i(s)\dot\gamma^j(s) - (H_\phi)_{ij}(\gamma_{\Pi}(s),\dot\gamma_{\Pi}(s))\dot\gamma_{\Pi}^i(s)\dot\gamma_{\Pi}^j(s)\nonumber\\
&=(H_\phi)_{ij}(\gamma(s),\dot\gamma(s))\dot\gamma^i(s)\dot\gamma^j(s) - (H_\phi)_{ij}(\gammap(s),\dot\gamma_{\Pi}(s))\dot\gamma^i(s)\dot\gamma^j(s)\nonumber\\
&\quad+(H_\phi)_{ij}(\gammap(s),\dot\gamma_{\Pi}(s))(\dot\gamma^i(s)+\dot\gamma_{\Pi}^i(s))(\dot\gamma^j(s)-\dot\gamma_{\Pi}^j(s)).
\eal
From
\eqref{controllo1}, recalling that $\gamma$ is parameterized
with Finslerian unit speed, using the fact that the second derivatives of $\phi$ are Lipschitz functions and the $\Gamma^k_{ij}(x,y)$ are smooth
on $TW\setminus 0$, we get
\bal\label{hij1}
\lefteqn{\big [(H_\phi)_{ij}(\gamma(s),\dot\gamma(s))- (H_\phi)_{ij}(\gammap(s),\dot\gamma_{\Pi}(s))\big]\dot\gamma^i(s)\dot\gamma^j(s)}&\nonumber\\
&\leq \|H_\phi(\gamma(s),\dot\gamma(s))- H_\phi(\gammap(s),\dot\gamma_{\Pi}(s))\| a^2\nonumber\\
&\leq a_1\big(|\gamma(s)-\gammap(s)|+ |\dot\gamma(s)-\dot\gamma_{\Pi}(s)|\big),\eal
where $\|\cdot\|$ denotes the norm on the space of bounded bilinear operator on $\R^n\times\R^n$.
As $\eta$ is a
$C^{2,1}$ map, we obtain \beq\label{crho}
|\gamma(s)-\gamma_{\Pi}(s)|=|\eta(0,\gamma(s))-\eta(\rho(s),\gamma(s))|\leq
a_2\rho(s).\eeq
Moreover,  from (\ref{controllo1})
\be  \label{n1}
|  (H_\phi)_{ij}(\gamma_\Pi (s),\dot\gamma_\Pi(s))(\dot\gamma^i(s)+\dot\gamma_{\Pi}^i(s))(\dot\gamma^j(s)-\dot\gamma_{\Pi}^j(s))  | \leq a_3  |  \dot \gamma (s) - \dot\gamma_{\Pi} (s) |
\ee
and  from \eqref{dotgammapi}
\begin{eqnarray}
\lefteqn{ | \dot\gamma(s)-\dot\gamma_{\Pi}(s) | }  \nonumber \\
   &  =  & |   \partial_x\eta(0,\gamma(s))[\dot\gamma(s)]-\partial_x\eta(\rho(s),\gamma(s))
             [\dot\gamma(s)]-\partial_t\eta(\rho(s),\gamma(s))\dot\rho(s) |
  \nonumber \\
    &   \leq    &  a_4 \rho (s) + a_5 | \dot \rho (s)  |   \label{cdotrho}.
\end{eqnarray}
Thus from \eqref{hij}, \eqref{hij1},
\eqref{crho}, \eqref{n1}, \eqref{cdotrho},  we easily get
\[\ddot\rho(s)\leq a_6 (\rho(s)+   |\dot\rho(s)  | ) \]
and, since $\rho\geq 0$,
$\rho(0)=\dot\rho(0)=0$,   Lemma \ref{grnew}
implies that $\rho=0$ on $[0,\sigma]$.
\end{proof}
\bere\label{geoconvequiv}
Since infinitesimal convexity with respect to the metric $F$ is equivalent to that for the reversed
 metric $\tilde F$ (see Remark~\ref{infconvequiv}), an analogous statement holds also for geodesics of $\tilde F$
(or in other words for the reversed curves obtained from the  geodesics of $(M,F)$) assuming local infinitesimal convexity of $\partial D$ with respect to $F$.
\ere
\begin{cor}\label{cor}
If $\partial D$ is infinitesimally convex then $\partial D$ is geometrically convex.
\end{cor}
\begin{proof}
Otherwise, a geodesic $\gamma: [0,1]\rightarrow \overline D$ with $\gamma(0), \gamma(1)\in D$ and $c\in]0,1[$ exist such that $\gamma(c)\in \partial D$ and $\gamma(]c,1])\subset D$. Necessarily, $\dot \gamma(c)\in T_p\partial D$, in contradiction with Proposition \ref{pA}.
\end{proof}
Using Proposition~\ref{pA}, the following lemma can be proved in a
standard way (cf. \cite[\S 4]{w}). We observe that differently
from  \cite{w}, where  strict convexity of the boundary is imposed,
we cannot state that the closure of $D\cap B^+(p,\delta)$ is  {\em
completely convex}.\footnote{We recall that the closure of a
subset $X$ of a Riemannian or Finsler manifold is said completely
convex if any two points in $\overline X$ can be connected by a
geodesic lying in $X$, with the possible exception of either 
one endpoint or both.} In fact infinitesimal convexity does not
exclude the case in which the points  $p_1, p_2\in \partial D\cap
B^+(p,\delta)$ are  connected by a geodesic included in $\partial
D$.
\begin{lem} \label{lA}
Assume that $\partial D$ is infinitesimally convex in a neighborhood $U$ of $p\in\partial D$. Then a small enough convex ball $B^+(p,\delta)$ exists such that $\partial D\cap B^+(p,\delta)\subset U$ and for each $p_1,p_2\in D\cap B^+(p,\delta)$ the (unique) geodesic in $B^+(p,\delta)$
 which connects $p_1$ with $p_2$ is included in $D$.
 \end{lem}
\begin{proof}
Set $C=D\cap B^+(p,\delta)$ and let $A$ be
the set of points $(p_1,p_2)\in C\times C$ that can be connected  by a geodesic having support in $D$. We are going to prove that $A=C\times C$. Since $C$ is connected, it is enough to prove that $A$ is non-empty, open and closed in $C\times C$. Clearly each couple $(p_1,p_1)\in C\times C$ can be connected by a constant geodesic, hence $A\neq\emptyset$. If $(p_1,p_2)\in A$ and $\gamma$ is the unique geodesic connecting them and whose inner points are in $D\cap B^+(p,\delta)$, by smooth dependence of geodesics by initial conditions, we can consider two small enough neighborhoods $U_1$ and $U_2$ of $p_1$ and, respectively, $p_2$ in $C$, such that the geodesic in $B^+(p,\delta)$ connecting $\bar p_1\in U_1$ to $\bar p_2\in U_2$ lies in a small neighborhood of $\gamma$, hence its  points are contained in $D$. Now let $(p_1,p_2)\in \overline A\subset C\times C$ and consider a sequence $(p^1_n, p^2_n)_n\subset A$  converging to $(p_1,p_2)$. Up to reparametrizations, the geodesics $\gamma_n$ connecting $p^1_n$ to $p^2_n$ converge uniformly to the geodesic $\gamma$ contained in $B^+(p,\delta)$ and connecting $p_1$ to $p_2$. Thus $\gamma$ lies in $\overline D$. As the points $p_1$ and $p_2$ are in $D$, from Proposition~\ref{pA}, $\gamma$ cannot be tangent to $\partial D$ at any of its inner points.
\end{proof}
\begin{proof}[\bf Proof of Theorem~\ref{mainmain}]
Assume by contradiction that $N$  is not locally convex at $p\in N$. Denoting by $D$ the inner domain of $N$,
then a sequence of tangent vectors
$v_n\in T_{p}N$ exists such that $v_n\rightarrow 0$ and each $p_n=\exp (v_n)$ or $q_n=\widetilde\exp(v_n)$ belongs to $D\cap U$.
To fix ideas, let us assume that $\big(p_n=\exp(v_n)\big)_n$ exists such that $p_n\in D\cap U$. From Remark~\ref{geoconvequiv} the other case can be proved in the same manner.

Let $B^+(p,\delta)$  be as in Lemma~\ref{lA}.
With no loss of generality, we can assume that all $p_n$ belong to
$B^+(p,\delta)\cap D$ hence, from Lemma~\ref{lA}, each unit speed geodesic $\alpha_n:[0,b_n]\rightarrow U$ which connects $p_n$ with the first point  $p_1$
of the sequence is included in $D\cap B^+(p,\delta)$. As $B^+(p,\delta)$ is relatively compact  the sequence of curves $(\alpha_n)_n$ uniformly converges
 to a curve
$\alpha: [0,b] \rightarrow \overline D\cap \overline B^+(p,\delta)$. In fact, $\alpha$
is also a geodesic which connects
$p$ with $p_1$ and, thus, it coincides with the original geodesic
$\gamma_1$, up to a reparameterization. Since $p_1\in D\cap U$, $\gamma_1$ must definitively leave the boundary $N$ of $D$. Thus, denoting by  $s_M\in[0,b_1[$ the maximum value of the parameter such that $\gamma_1(s)\in N$,  it must be $\dot \gamma_1(s_M)\in T_{\gamma_1(s_M)}N$ and $\gamma_1(]s_M,b])\subset U\cap D$. On the other hand,  from Proposition~\ref{pA}, $\dot \gamma_1(s_M)\not\in T_{\gamma_1(s_M)}N$. This contradiction concludes the proof.
\end{proof}

\section{Convexity of a Domain}\label{s3}
In the following,  we denote by $D$  a $C^{2,1}_{\rm loc}$ domain
of a Finsler manifold $(M,F)$. Let us set, for any $p,q\in D$,
\bmln \Omega(p,q;D) = \Big\{\gam:[0,1]\rightarrow D\mid\gam\text{
absolutely  continuous, }\\  \inte h(\gam)[\dot \gam,\dot
\gam]\;\de s< + \infty, \gam(0)=p, \gam(1)=q\Big\}, \emln where
$h$ is any complete auxiliary Riemannian metric on $M$. It is well
known that $\Omega(p,q;D)$ is an infinite dimensional manifold
(cf. e.g. \cite{k}). Let us denote the function
$F^2$ by $G$; then consider the functional \be\label{J}
J\colon\Om\to\R,\quad\quad\quad J(\gam)=\frac{1}{2}\inte
G(\gam,\dot \gam)\; \de s \ee which is a $C^1$ functional with
locally Lipschitz differential (see \cite[Theorem 4.1]{m1}). A
critical point $\gamma$ of $J$ is a curve $\gamma\in\Om$ such that
$\de J(\gamma)=0$. As recalled in Section \ref{s1}, the critical
points of $J$ on $\Omega(p,q;M)$ are smooth  curves and are all
and only the geodesics of the Finsler manifold $(M,F)$ connecting
the point $p$ to the point $q$.  Thus, looking for geodesics
joining two arbitrary points $p,q\in D$ and having support in $D$,
is equivalent to seek for the critical points of $J$. From the
viewpoint of critical point theory, a natural condition to assume on $J$
would be to satisfy the so-called Palais-Smale condition. In fact,
it is proved in \cite{cys} that this condition is fulfilled when
the symmetrized balls of $(M,F)$ are compact. But in our setting,
as $D$ is an open subset of $M$, the manifold \Om is not complete
in any reasonable sense. As a consequence,  Palais-Smale sequences
may converge to  curves touching the boundary $\partial D$. In
order to overcome this difficulty, we use a {\em penalization
method}.

For any $\eps\in ]0,1]$, we consider on $\Omega(p,q;D)$
the functional
\be\label{Jeps}
J_\eps(\gam) =
J(\gam) + \inte{\frac \eps{\phi^2(\gam)}}\, \de s
\ee
where $J$ and $\phi$ were respectively introduced in (\ref{J}), (\ref{defbt}).

The presence of the penalizing term
\[\gamma\in \Om\mapsto \inte\frac{\eps}{\phi^2(\gamma)}\de s,\]
combined with the lack of regularity of $G$ on the zero section, makes the study of regularity and existence of critical points of $J_\eps$ more subtle than the one for the unpenalized functional $J$ and that of the analogous Riemannian problem.

\begin{lem}\label{loose}
For any $\eps\in ]0,1]$, let $\gam_\eps\in \Om$ be a critical point of $J_\eps$ in (\ref{Jeps}). Then $\gam_\eps$ is $C^1$ and, for any $\bar s\in [0,1]$ such that $\dot\gam_\eps(\bar s)\not=0$,
said $(U,\Phi)$ a chart of $M$ such that $\gam_\eps(\bar s)\in U$, $\gamma_\eps$ is twice differentiable on the open subset $\gam_\eps^{-1}(U)$ and there
it satisfies
\be\label{twice}
\ddot\gam_\eps^i(s) +\Gamma^i_{jk}(\gamma_\eps(s),\dot\gamma_\eps(s))\dot\gamma_\eps^j(s)\dot\gamma_\eps^k(s)= -\frac{2\eps}{\phi^3(\gamma_{\eps}(s))}\partial_{x^k}\phi(\gam_\eps(s))g^{ki}(\gamma_\eps(s),\dot\gamma_\eps(s)).
\ee
Moreover a constant $E_\eps({\gamma_{\eps}})\in\R$ exists such that
\be\label{ene}
E_\eps({\gamma_{\eps}})=\frac 1 2G(\gamma_{\eps},\dot\gamma_{\eps})-\frac{\eps}{\phi^2(\gamma_{\eps})} \quad\hbox{\rm  on } [0,1].\ee
\end{lem}
\begin{proof}
Let us consider a finite covering of the range of $\gamma_\eps$ made by local charts $(U_k,\Phi_k)$ of $M$ and let $(TU_k,T\Phi_k)$ be the corresponding charts of $TM$. Let us define the intervals $I_k=\gamma_\eps^{-1}(U_k)=]s_k,s_{k+1}[\subset [0,1]$.
As $\gam_\eps\in \Om$ is a critical point of $J_\eps$, if $Z\in T_{\gam_\eps}\Om$ has compact support in the interval $I_k$, we have
\bal\label{pc}
\lefteqn{\de J_\eps(\gamma_\eps)[Z]=}&\nonumber\\
&=\frac{1}{2}\int_{I_k} \big(\partial_x
G(\gamma_\eps,\dot\gamma_\eps)[Z]+\partial_yG(\gamma_\eps,\dot\gamma_\eps)[\dot
Z]\big)\, \de s - \inte{\frac {2\eps}{\phi^3(\gamma_\eps)}}\partial_x\phi(\gamma_\eps)[Z]\, \de s =0.\eal
With an abuse of notation,  in the integral  above we denoted by $G$ the function $G\circ(T\Phi_k)^{-1}$ and by $\gamma_{\eps}$ the curve $\Phi_k\circ\gamma_{\eps}$
(so that the derivatives $\partial_x G$ and $\partial_y G$ have to be intended as the corresponding derivatives in $\Phi_k(U_k)\times\R^n$ of the function $G\circ(T\Phi_k)^{-1}$).

Evaluating (\ref{pc})
on any smooth vector field $Z$ along $\gam_\eps$ with compact support on the interval
$I_k$, we get
\be\label{dubois}
\int_{I_k}\left( H
+\frac 12\partial_yG(\gamma_\eps,\dot\gamma_\eps)\right)[\dot  Z]\, \de s=0, \ee where
$H=H(s)$ is the covector field along $\gamma_\eps$ defined as
\[
H(s)=-\int_{s_k}^s\left(\frac 12\partial_x G(\gamma_\eps,\dot\gamma_\eps) - {\frac{2\eps}{\phi^3(\gamma_\eps)}}\partial_x\phi(\gamma_\eps)\right)\de t.
\]
Equation
(\ref{dubois}) implies that a constant  covector
$W\in\mathbb{R}^n$, with $n=\dim M$, exists such that
\begin{equation}\label{integraloperator}
H(s)+\frac 12\partial_yG(\gamma_\eps(s),\dot\gamma_\eps(s))=W \quad\hbox{ a.e. on } I_k;
\end{equation}
since $H$ is continuous, the function $s\in I_k\mapsto \partial_yG(\gamma_\eps,\dot\gamma_\eps)$ is also continuous.
Now fix $x\in D$ and consider  the map $\mathcal L_x\colon y\in \R^n \mapsto \partial_y G(x,y)\in R^n$.  It can be proved that
$\mathcal L_x$ is a homeomorphism of $\R^n$, hence  the function $s\in I_k\mapsto \mathcal L^{-1}_{\gamma_\eps(s)}\circ\mathcal L_{\gamma_\eps(s)}(\dot\gamma_\eps(s))=\mathcal L^{-1}_{\gamma_\eps(s)}\big(\partial_yG(\gamma_\eps(s),\dot\gamma_\eps(s))\big)=\dot\gamma_\eps(s)$ is continuous, thus $\gamma_\eps$ is a $C^1$ curve (cf. \cite[Proposition 2.3]{cym} for details).
From  (\ref{integraloperator}) and the fact that $G$ is fiberwise strictly convex, as in \cite[Proposition 4.2]{bgh}  the implicit function theorem implies  that $\gamma_\eps$ is actually twice differentiable  at each $s$ where $\dot\gamma_\eps(s)\neq 0$. Thus we infer (\ref{twice}) (recall Eqs.~ \eqref{spray}--\eqref{eqgeo}).
Let us consider now the open subsets $A_{k,\eps}= \{s\in I_k\mid\dot\gamma_\eps(s)\neq 0\}$. From \eqref{integraloperator}, we get
\[
\frac{\de}{\de s}\frac 12\partial_yG(\gamma_\eps,\dot\gamma_\eps)=\frac 12\partial_xG(\gamma_\eps,\dot\gamma_\eps) - {\frac {2\eps}{\phi^3(\gamma_\eps)}}\partial_x\phi(\gamma_\eps),
\]
and then the energy $E_{k,\eps}(\gamma_\eps)=\frac 12\partial_y G(\gamma_\eps,\dot\gamma_\eps)[\dot\gamma_\eps]-\frac 12G(\gamma_\eps,\dot\gamma_\eps) - {\frac{\eps}{\phi^2(\gamma_\eps)}}$ is constant on every connected component of $A_{k,\eps}$. As $G$ is positively homogeneous of degree $2$, from Euler's theorem:
$$E_{k,\eps}(\gamma_\eps)=\frac 12 G(\gamma_\eps,\dot\gamma_\eps)-{\frac{\eps}{\phi^2(\gamma_\eps)}}, \quad\hbox{ on each connected component of } A_{k,\eps}.$$
Since the function
$s\in I_k\mapsto \frac 12 G(\gamma_{\eps},\dot\gamma_{\eps})-\frac{\eps}{\phi^2(\gamma_{\eps})}$ is continuous,  it has to be
$E_{k,\eps}(\gamma_\eps)=-\frac{\eps}{\phi^2(\gamma_{\eps})}$ on $I_k\setminus A_{k,\eps}$. By standard arguments, constants
$E_{k,\eps}(\gamma_\eps)$ must agree, hence a real constant $E_\eps({\gamma_{\eps}})$ exists such that
(\ref{ene}) holds.
\end{proof}
To prove that functionals $J_\eps$ satisfy the Palais-Smale
condition, we adapt Gordon's lemma, (cf. e.g. \cite{b2}), which
essentially is the core of the penalization technique. We report
here the proof for the reader's convenience.
\begin{lem}\label{gordon}
Let $D$ be a domain of a Finsler manifold $(M,F)$ and assume that  $\overline B_s(p,r)\cap\overline D$ is compact, for any closed  ball $\overline B_s(p,r)\subset M$.
Let $(\gam_m)_m$ be a sequence in $\Om$
such that
\be \label{gg1}
\sup_{m\in \N}\inte G(\gam_m,\dot \gam_m)\, \de s =k^2< +\infty
\ee
for some $k>0$, and assume that $(s_m)_m$ is a sequence in $[0,1]$ such that
\[
\lim_{m\rightarrow +\infty} \phi(\gam_m(s_m)) = 0.
\]
Then
\[\lim_{m\rightarrow  +\infty}
\displaystyle{\inte{\frac 1{\phi^2(\gam_m)}}\, \de s = +\infty}.
\]
\end{lem}

\begin{proof}
Observe that  the supports of the curves
$\gam_m$ are contained in the intersection $\overline B^+(p,k)\cap \overline B^-(q,k)\cap \overline D$, which is a compact subset of $\overline D$ in our assumptions (see Remark~\ref{sym}). Indeed by (\ref{gg1}), for each $s\in[0,1]$, we have
\baln
&\de(p,\gam_m(s))\leq \int_0^s F(\gam_m,\dot\gam_m)\, \de t\leq \inte F(\gam_m,\dot\gam_m)\, \de t\leq \left(\inte G(\gam_m,\dot\gam_m)\, \de t\right)^{1/2}\leq  k\\
\intertext{and likewise} &\de(\gam_m(s),q)\leq \int_s^1
F(\gam_m,\dot\gam_m)\, \de t\leq \inte F(\gam_m,\dot\gam_m)\, \de t\leq
\left(\inte G(\gam_m,\dot\gam_m)\, \de t\right)^{1/2}\leq  k. \ealn
Hence a positive constant $C_1$ exists such that
\beq\label{c1c2} \frac{1}{C_1}|y|_{h(x)}^2\leq G(x,y)\leq
C_1|y|_{h(x)}^2, \eeq for every  $x\in K$ and for every $y\in T_x
M$. From H\"older inequality and \eqref{c1c2} we have that for
each $s\in ]s_m, 1]$ \beq\label{holder}
\phi(\gam_m(s))-\phi(\gam_m(s_m))=\int_{s_m}^s
h(\gam_m)[\nabla^0\phi(\gam_m),\dot\gam_m]\, \de t\leq
C_2(s-s_m)^{1/2} \eeq where $C_2$ is a positive constant depending
on $K$, but independent from $m$. From \eqref{holder} we get
\beq\label{f2f3}
\frac{1}{\phi^2(\gam_m(s))}\geq \frac{1}{2\big(C_2^2(s-s_m)+\phi^2(\gam_m(s_m))\big)}.\eeq
Moreover from \eqref{holder}, recalling that  for each $m\in \N,\ \gam_m(1)=q$, we deduce that a positive constant $C_3$ exists such that $1-s_m\geq C_3$ for all $m\in\N$. Therefore  integrating both hand sides of \eqref{f2f3} on the interval $[s_m,1]$ we get the thesis.
\end{proof}
Let $\mathcal M$ be a Banach manifold. We recall that a $C^1$ functional $f\colon \mathcal M\to \R$ satisfies the Palais-Smale condition if every sequence $(x_m)_m$ such that $\big(f(x_m)\big)_m$ is bounded and $\de f(x_m)\to 0$, as $m\to +\infty$, admits a converging subsequence.

\begin{prop}\label{pec}
Under the assumptions of Theorem~\ref{main}, then
\begin{list}{(\roman{enumi})}{\usecounter{enumi}\labelwidth  5em
\itemsep 0pt \parsep 0pt}
\item
for any
$\eps \in]0,1]$ and for any $c \in  \R$, the sublevels
$J_{\eps}^c= \{x\in \Om\mid J_\eps(x)\leq c\}$
are complete metric subspaces of $\Om$;
\item
for any $\eps\in ]0,1]$, $J_\epsilon$ satisfies the Palais-Smale condition.
\end{list}
\end{prop}
\begin{proof}
(i) Fix $\eps\in ]0,1]$ and $c\in\R$. Let $(\gam_m)_m$ be a Cauchy sequence
in $J_\eps^c$; then it is a Cauchy sequence also in $\Omega(p,q;M)$ and it uniformly converges
to a curve $\gam$ with support in $\overline D$.
By Lemma \ref{gordon} it follows
\[
d = \inf\displaystyle{\left\{\phi(\gam_m(s)) | s\in [0,1], m \in \N\right\}} > 0.
\]
Thus $\phi(\gam(s))\geq d$ for any $s\in [0,1]$, hence
$\gam \in \Om$ and, by the continuity of $J_\eps$, $J_\eps(\gamma)\leq c$,
as required.

\noindent
(ii) We can adapt the proof of the Palais-Smale condition for the unperturbed functional $J$ on $\Omega(p,q;M)$ in  \cite[Theorem 3.1]{cym}.   Indeed, reasoning as in in the first part of the proof of Lemma~\ref{gordon} we have that, if $(\gam_m)_m$ is a Palais-Smale sequence, then a compact subset $K\subset\overline D$ containing the supports of the curves $\gam_m$ exists and \eqref{c1c2} holds. Therefore for all $s_1, s_2\in[0,1]$ a positive constant $C$ exists such that
\[\de_h(\gam_m(s_1),\gam_m(s_2))\leq \int_{s_1}^{s_2}|\dot\gam_m|_{h}\de s\leq C|s_1-s_2|^{1/2} \quad \text{for all $m\in\N$,}\]
where $\de_h$ is the distance associated to the auxiliary Riemannian metric $h$.
Hence, from the Ascoli-Arzel\`a theorem, there exists  a
subsequence,  denoted again by $(\gam_m)_m$, which uniformly
converges to a curve $\gam$ whose support is in $\overline D$. From Lemma~\ref{gordon}, $\gam([0,1])\subset D$, otherwise $J_\eps(\gam_m)\to +\infty$, in contradiction with the assumption that $(\gamma_m)_m$ is a Palais-Smale sequence.  Now consider a smooth curve $w\in\Om$ which approximates $\gamma$ and a natural  chart  $(\mathcal O_w,\exp_w^{-1})$, centered at $w$, of the manifold \Om (cf. \cite[Corollary 2.3.15]{k}). Set $\xi_m=\exp_w^{-1}(\gam_m)$.
We have
\bal\label{pseps}
\lefteqn{\de J_\eps(\gamma_m)\big[\de\exp_w(\xi_m)[\xi_m-\xi_n]\big]=}&\nonumber\\
&\de J(\gamma_m)\big[\de\exp_w(\xi_m)[\xi_m-\xi_n]\big]-2\eps\inte\frac{h(\gam_m)\big[\nabla^h\phi(\gam_m),\de\exp_w(\xi_m)[\xi_m-\xi_n]\big]}{\phi^3(\gam_m)}\, \de s.\eal
As $(\gam_m)_m$ uniformly converges to $\gam$, $\xi_m-\xi_n\to 0$ uniformly as $m,n\to+\infty$, hence the second  term in the right-hand  side of \eqref{pseps} goes to zero as $m,n\to+\infty$. Since $(\gam_m)_{m}$ is a Palais-Smale sequence, also the left-hand side goes to zero and then
\beq\label{psuneps}
\de J(\gamma_m)\big[\de\exp_w(\xi_m)[\xi_m-\xi_n]\big]\To 0,
\eeq as $m,n\to+\infty$.  From \eqref{psuneps}, the rest of the proof follows step by step that in \cite[Theorem 3.1]{cym}.
\end{proof}
\bere\label{bound}
By Proposition \ref{pec}, for any $\eps\in]0,1]$, $J_\eps$ has a minimum
point $\gamma_\eps\in\Om$; then $k>0$ exists such
that
\be\label{up}
J_\eps(\gam_\eps)\leq k\quad  \text{for all $\eps\in]0,1],$}
\ee
since $J_\eps(\gam_\eps)\leq J_\eps(\gam_1)\leq J_1(\gam_1)$.
Moreover, from (\ref{ene}) we get
$$
E_\eps(\gam_\eps) =
J_\eps(\gam_\eps) - 2\inte{\frac \eps{\phi^2(\gam_\eps)}}\, \de s \leq k
\quad \text{for all $\eps\in]0,1]$}
$$
thus
\be\label{cota}
{\frac 12}G(\gam_\eps(s),\dot\gam_\eps(s))
\leq k + {\frac\eps{\phi^2(\gam_\eps(s))}}
\quad \text{for all $\eps\in]0,1]$ and $s\in [0,1]$.}
\ee
\ere
Next we give an a priori estimate about the critical points of the functionals $J_\eps$.
\begin{lem}\label{cruciale}
Under the assumptions of Theorem~\ref{main}, let $(\gam_\eps)_{\eps \in ]0,1]}$ be a family in $\Om$ such that, for any $\eps\in ]0,1]$, $\gam_\eps$ is a critical
point of $J_\eps$ and let $k\in\R$ be such that
(\ref{up}) holds.
Then, set
$$
\lambda_\eps(s) = {\frac{2\eps}{\phi^3(\gam_\eps(s))}} \quad \text{for all $\eps\in]0,1]$ and $s\in [0,1]$,}
$$
$\eps_0\in ]0,1]$ exists such that  $(\|\lambda_\eps\|_\infty)_{\eps \in ]0,\eps_0]}$ is bounded, where
\[\|\lambda_\eps\|_\infty=\max_{s\in[0,1]}\lambda_\eps(s).\]
\end{lem}
\begin{proof}
Let $(\gam_\eps)_{\eps \in ]0,1]}$ be a family
of critical points of
$J_{\eps}$ satisfying (\ref{up}) and let
us set for any $\eps\in ]0,1], s \in [0,1],$
$\rho_\eps(s) = \phi(\gam_\eps(s))$ and
$\rho_\eps(s_\eps) = {\displaystyle \min_{s\in [0,1]}\rho_\eps(s)}$.
It is enough to prove the thesis when
\[\lim_{m\rightarrow +\infty}\rho_{\eps_m}(s_{\eps_m}) = 0,
\]
where $(\eps_m)_m$ is any  infinitesimal and decreasing sequence
in $]0,1]$.
Note also that, reasoning as in the first part of the proof of Lemma~\ref{gordon}, the supports of the curves $(\gam_\eps)_{\eps \in ]0,1]}$ are contained in a compact subset.

We distinguish the following two cases.

\noindent
{\em First case:}  Let $s_\eps\in A_\eps= \{s\in [0,1]\mid\dot\gamma_\eps(s)\neq 0\}$; then from Lemma~\ref{loose},
$\gamma_\eps$ is $C^2$ in a neighborhood of $s_{\eps}$ and
$\dot \rho_\eps(s_\eps)=0$, $\ddot \rho_\eps(s_\eps)\geq 0$, where
$$
\ddot \rho_\eps(s)=
\frac{\partial^2\phi}{\partial x^i\partial x^j}(\gamma_\eps(s))\dot\gamma_\eps^i(s)\dot\gamma_\eps^j(s) + \frac{\partial\phi}{\partial x^i}(\gam_\eps(s))\ddot\gamma_\eps^i(s).
$$
In this case the proof is essentially the same as for domains in a Riemannian manifold (cf. \cite{gm}).
Indeed, taking into account that (\ref{twice}) holds on a neighborhood of $s_\eps$, we get:
\bml
0\leq\ddot \rho_\eps(s_\eps)=
\frac{\partial^2\phi}{\partial x^i\partial x^j}(\gamma_\eps(s_\eps))\dot\gamma_\eps^i(s_\eps)\dot\gamma_\eps^j(s_\eps)\\
 -\frac{\partial\phi}{\partial x^i}(\gam_\eps(s_\eps))
\Gamma^i_{jk}(\gamma_\eps(s_\eps),\dot\gamma_\eps(s_\eps))\dot\gamma_\eps^j(s_\eps)\dot\gamma_\eps^k(s_\eps)\\
 -\frac{2\eps}{\phi^3(\gamma_{\eps}(s_\eps))}\frac{\partial\phi}{\partial x^k}(\gam_\eps(s_\eps)) \frac{\partial\phi}{\partial x^i}(\gam_\eps(s_\eps))g^{ki}(\gamma_\eps(s_\eps),\dot\gamma_\eps(s_\eps)).\label{ddotrho}
\eml
As the components of the Chern connection are positively homogeneous of degree $0$ with respect to $y$, $\Gamma^i_{jk}(\gamma_\eps(s_\eps),\dot\gamma_\eps(s_\eps))=
\Gamma^i_{jk}\left(\gamma_\eps(s_\eps),\frac{\dot\gamma_\eps(s_\eps)}{|\dot\gamma_\eps(s_\eps)|}\right)$, by the fact that the supports of the curves $\gamma_\eps$ are contained in a compact subset of $M$, the first two terms in the right-hand side of \eqref{ddotrho} can be bounded by $k_1G(\gam_\eps(s_\eps),\dot\gam_\eps(s_\eps))$, for a positive constant $k_1$.
Analogously, since $0$ is a regular value for $\phi$ and the matrix $[g^{ki}(x,y)]$ is positive definite, for all $(x,y)\in TM\setminus 0$, and positively homogeneous of degree $0$, a positive constant $k_2$ exists such that
\[\frac{\partial\phi}{\partial x^k}(\gam_\eps(s_\eps)) \frac{\partial\phi}{\partial x^i}(\gam_\eps(s_\eps))g^{ki}(\gamma_\eps(s_\eps),\dot\gamma_\eps(s_\eps))>k_2.\]
Hence, from \eqref{ddotrho} and (\ref{cota}) we get
\baln
0&\leq k_1
G(\gam_\eps(s_\eps),\dot\gam_\eps(s_\eps))
- k_2 {\frac{2\eps}{\phi^3(\gam_\eps(s_\eps))}}\\
&\leq k_1\left(2k + {\frac{2\eps}{\phi^2(\gam_\eps(s_\eps))}}\right)
- k_2 {\frac{2\eps}{\phi^3(\gam_\eps(s_\eps))}}
\ealn
and then
\be\label{chiave}
\frac{\eps}{\phi^3(\gam_{\eps}(s_{\eps}))}\leq
        c\left(k+\frac{\eps}{\phi^2(\gam_{\eps}(s_{\eps}))}\right)\ee
and the thesis follows.

\noindent
{\em Second case:}
Let $s_{\eps}\in B_{\eps}=[0,1]\setminus A_{\eps}$; we can further distinguish the following possibilities:
\begin{itemize}
\item [$(a)$] $s_{\eps}$ is an isolated point of $B_{\eps}$;
\item [$(b)$] $s_{\eps}$ is an accumulation point not in the interior of $B_{\eps}$;
\item [$(c)$] $s_{\eps}$ is in the interior of $B_{\eps}$.
\end{itemize}

\noindent
In case $(a)$, recalling that  $\rho_{\eps}$ is $C^1$, we have $\dot\rho_\eps(s_\eps)=0$ and a neighborhood $U(s_\eps)$ of $s_\eps$ exists such that on $U(s_\eps)$,  $\dot\rho_\eps(s)\geq 0$ for $s>s_\eps$, $\dot\rho_\eps(s)\leq 0$ for $s<s_\eps$;  hence, for each $s\in U(s_\eps)\setminus \{s_\eps\}$,  $\ddot\rho_\eps(s)$ exists and, according to the mean value theorem, $\ddot\rho_\eps(s)\geq 0$.
Let us then consider a sequence $(s_{\eps,m})_m$ in $U(s_\eps)\setminus \{s_\eps\}$ converging to $s_{\eps}$; reasoning as in the first step of the proof we get for any $m\in\N$:
\[
\frac{\eps}{\phi^3(\gam_{\eps}(s_{\eps,m}))}\leq
        c\left(k+\frac{\eps}{\phi^2(\gam_{\eps}(s_{\eps,m}))}\right)\]
and, passing to the limit, we get again  formula (\ref{chiave}) for $s_\eps$.

\noindent
In case $(b)$ there exists in $B_{\eps}$ a strictly monotone sequence, say strictly increasing, $(s^1_{\eps,m})_m$ converging to $s_{\eps}$.
We know that $\dot\rho_{\eps}(s^1_{\eps,m})=0$ for any $m\in\N$. Applying Rolle's theorem to function $\dot\rho_\eps$ on each $[s^1_{\eps,m-1}, s^1_{\eps, m}]$, we get another sequence, say $(s^2_{\eps,m})$, in $A_{\eps}$ which tends to $s_{\eps}$ and such that $\ddot\rho_\eps(s^2_{\eps,m})=0$
for any $m\in\N$. Then as in the first step we get:
\[
\frac{\eps}{\phi^3(\gam_{\eps}(s^2_{\eps,m}))}\leq
        c\left(k+\frac{\eps}{\phi^2(\gam_{\eps}(s^2_{\eps,m}))}\right)\]
and, passing to the limit, we still get formula (\ref{chiave}).

\noindent
Case $(c)$ has to be ruled out: indeed the interior of $B_{\eps}$ is empty. For if, assume that a neighborhood $U(s_\eps)$ in $B_{\eps}$ exists on which $\gamma_{\eps}$ has zero derivative.
Take then a vector $z\in T_{\gamma_{\eps}(s_{\eps})}D$ such that
$\partial_x\phi(\gamma_{\eps}(s_{\eps}))[z]>0$ and a variation vector field $Z$ such that $Z(s_{\eps})=z$, with support in  $U(s_\eps)$, in such a way that $\partial_x\phi(\gamma_{\eps})[Z]\geq 0$ on $U(s_\eps)$. As $\gamma_{\eps}$ is a critical point of $J_{\eps}$ we get a contradiction. Indeed, as $\partial_x G(\gamma_\eps,\dot\gamma_\eps)=0$ and $\partial_y G(\gamma_\eps,\dot\gamma_\eps)=0$, it follows
\baln
\lefteqn{0=\de J_\eps(\gamma_\eps)[Z]=}&\\
&=\frac{1}{2}\int_{U(s_\eps)} \big(\partial_x
G(\gamma_\eps,\dot\gamma_\eps)[Z]+\partial_yG(\gamma_\eps,\dot\gamma_\eps)[\dot
Z]\big)\, \de s - 2\eps\int_{U(s_\eps)}\frac{\partial_x\phi(\gamma_{\eps})[Z]}{\phi^3(\gamma_{\eps})}\, \de s\\
&=- 2\eps\int_{U(s_\eps)}\frac{\partial_x\phi(\gamma_{\eps})[Z]}{\phi^3(\gamma_{\eps})}\, \de s <0.
\ealn
\end{proof}
The lemma above is fundamental in order to conclude the limit process. Indeed we can state the following proposition.
\begin{prop}\label{ben}
Let $(\gam_\eps)_{\eps \in ]0,1]}$ be a family in $\Om$ such that for any $\eps\in ]0,1]$ $\gam_\eps$ is a critical
point of $J_\eps$ and let $k>0$ be such that
(\ref{up}) holds.
Then, a subsequence $(\eps_m)_m$ in $]0,1]$ exists such that
\begin{itemize}
\item [$(1)$] $(\gam_{\eps_m})_m$ strongly converges to a curve $\gam\in \Omega(p,q;M)$ whose support is contained in $\overline D$;
\item [$(2)$] $(\lambda_{\eps_m})_m$ weakly converges to $\lambda\in L^2([0,1],\R)$;
\item [$(3)$]
the limit curve $\gamma$ is $C^1$ and, for any $\bar s\in [0,1]$ such that $\dot\gam(\bar s)\not=0$,
said $(U,\Phi)$ a chart of $M$ such that $\gam(\bar s)\in U$, $\gamma$ has $H^{2,2}$-regularity on an open subset of $\gam^{-1}(U)$ containing the point $\bar s$ and there
it satisfies a.e. the equations
\be\label{weakeq}
\ddot\gam^i +\Gamma^i_{jk}(\gamma,\dot\gamma)\dot\gamma^j\dot\gamma^k= -\lambda\partial_{x^k}\phi(\gam)g^{ki}(\gam,\dot\gam).\ee
\end{itemize}
\end{prop}
\begin{proof}
Statements $(1)$ and $(2)$ respectively follow  by an argument analogous to that used in Proposition \ref{pec} and by Lemma  \ref{cruciale}. Let us prove \eqref{weakeq}. As observed in Section \ref{s1}, we cannot proceed as
in previous references on the topic, because therein, by the Nash  embedding theorem, the Riemannian  manifold $M$ is treated as a closed submanifold of an Euclidean space $\R^N$ (see \cite{muller} for the existence of a closed isometric embedding) and some arguments  based  on the vector space structure of  the Hilbert space $H^{1,2}([0,1],\R^N)$ are used  (cf. \cite[Lemma 4.6, Lemma 4.7]{gm}). Instead,
our proof relies on a local representation of the Lagrangian $G$ as in \cite{b1}.

Let  ${\mathcal A}=\{(V_i,\Phi_i)\}$ be  a smooth  atlas of $M$ and   $T{\mathcal A}=\{(TV_i,T\Phi_i)\}$ the corresponding atlas of $TM$.
Let us consider, for any $V_i\in \mathcal A$, the Lagrangian $G_{V_i}:\Phi_i(V_i)\times\R^n\rightarrow \R$, $G_{V_i}(q,v)=G\circ (T\Phi_i)^{-1}(q,v)$, where $n=\dim M$. Note that the Lagrangians  $G_{V_i}$ are positively homogeneous of degree $2$ with respect to the variable $v\in\R^n$. Let us consider the functionals
$$
j_{V_i}(q)=\frac12\int_{I_q}{G}_{V_i}(q(s),\dot q(s))\, \de s,
$$
where $q$ is any curve in $\Phi_i(V_i)$, namely $q\colon I_q\subset \R\to \Phi_i(V_i)$.
Then we set
\[j_{\eps_m,V_i}(q)=j_{V_i}(q)+\int_{I_q}\frac{\eps_m}{\phi_{V_i}^2(q(s))}\, \de s,\]
where $\phi_{V_i}=\phi\circ (\Phi_i)^{-1}\colon \Phi_i(V_i)\to\R$.
Now let  $\{(U_k, \Phi_k)\}$ be a finite covering of $\gamma$ such that, for each $k$, $(U_k,\Phi_k)\in\mathcal A$. Let $I_k=\gamma^{-1}(U_k)\subset [0,1]$; clearly for $m$ large enough it is also $\gam_{\eps_m}(I_k)\subset U_k$.
We have
\beq\label{localizeJ}
J_{\eps_m}(\gamma_{\eps_m})=\sum_k j_{\eps_m, U_k}(q_{mk})
\eeq
where $(q_{mk},\dot q_{mk})=T\Phi_k(\dot\gamma_{{\eps_m}|_{I_k}})$.
Let $(q_{k},Z_{k})=T\Phi_k(Z|_{I_k})$, where $Z\in T_\gam(p,q;M)$. Clearly we can  view  the vector field $s\in I_k\mapsto Z_k(s)\in \R^n$ along $q_k$ as a vector field along the curve $q_{mk}$.
Then we have
\bal
\lefteqn{\de j_{\eps_m,U_k}(q_{mk})[Z_k]=}&\nonumber\\
&=\frac 12\int_{I_k} \big(\partial_q
G_{U_k}(q_{mk},\dot q_{mk})[Z_k]+\partial_vG_{U_k}(q_{mk},\dot q_{mk})[\dot
Z_k]\big)\, \de s\nonumber\\
&\quad\quad - 2\eps_m\int_{I_k}\frac{\partial_q \phi_{U_k}(q_{mk})[Z_k]}{\phi_{U_k}^3(q_{mk})}\, \de s\nonumber\\
&=\frac 12\int_{I_k} \left(|\dot q_{mk}|^2\psi_{mk} + |\dot q_{mk}|\chi_{mk}\right)\, \de s
-\int_{I_k} {\frac{2\eps_m}{\phi_{U_k}^3(\gam_{\eps_m}(s))}}\partial_q\phi_{U_k}(q_{mk})[Z_k]\, \de s,\label{locdiff}
\eal
where for a.e. $s\in I_k$,  $\psi_{mk}= \partial_q G_{U_k}(q_{mk}, \frac{\dot q_{mk}}{|\dot q_{mk}|})[Z_k]$,
$\chi_{mk}= \partial_v G_{U_k}(q_{mk}, \frac{\dot q_{mk}}{|\dot q_{mk}|})[\dot Z_k]$.
Observe that, for each $k\in\N$, the sequences of functions $(\psi_{mk})_m$ and $(\chi_{mk})_m$ are bounded in $L^{\infty}(I_k,\R)$ and
since $q_{km}\to q_k$ in $H^1(I_k,\Phi_k(U_k))$, by
the Lebesgue dominated convergence theorem we deduce that the first integral in \eqref{locdiff} converges to
\[
\frac 12\int_{I_k} \big(\partial_q
G_{U_k}(q_k,\dot q_k)[Z_k]+\partial_vG_{U_k}(q_k,\dot q_k)[\dot
Z_k]\big)\, \de s,
\]
as $m\to+\infty$. On the other hand from the weak convergence of $(\lambda_{\eps_m})_m$ to $\lambda$ and the uniform convergence of $(q_{mk})_m$ to $q_k$, we get
\[
\int_{I_k} \lambda_{\eps_m}(s)\partial_q \phi_{U_k}(q_{mk})[Z_k]\, \de s\longrightarrow \int_{I_k} \lambda(s)\partial_q \phi_{U_k}(q_k)[Z_k]\, \de s,
\quad \text{as $m\rightarrow +\infty$.}
\]
Summing over $k$,  from \eqref{localizeJ}, we
get
\beq\label{sumdiff}
\sum_k \de j_{\eps_m,U_k}(q_{mk})[Z_k]=\de J_{\eps_m}(\gamma_m)[Z_m]=0,\eeq
where $Z_m$ is the vector field along $\gam_m$ obtained patching together the vector fields $(T\Phi_k)^{-1}(Z_k)$.
Thus, from \eqref{locdiff}--\eqref{sumdiff}
we  deduce that $\gam$ satisfies the following equation
\[\frac{1}{2}\inte \big(\partial_x
G(\gamma,\dot\gamma)[Z]+\partial_yG(\gamma,\dot\gamma)[\dot
Z]\big)\, \de s - \inte\lambda(s)\partial_x\phi(\gamma)[Z]\, \de s=0
\]
(here we use the same abuse of notation as in  the proof of Lemma~\ref{loose}).
Consider now any smooth vector field $Z$ along $\gam$ with compact support in the interval
$I_k=\gamma^{-1}(U_k)=\, ]s_k,s_{k+1}[\subset [0,1]$; then
\[
\int_{I_k}\left( H
+\frac 12\partial_yG(\gamma,\dot\gamma)\right)[\dot  Z]\, \de s=0, \] where
$H=H(s)$ is the covector field along $\gamma$ defined as
\[
H(s)=-\int_{s_k}^s\left(\frac 12\partial_x G(\gamma,\dot\gamma) - \lambda\partial_x\phi(\gamma)\right)\de t.
\]
Reasoning  as in Lemma \ref{loose}, we get that $\gamma$ is a $C^1$ curve.
Since  $\lambda\in L^2([0,1],\R)$, as in the proof of Lemma~\ref{loose},  now we get that 
$\gam$ has $H^{2,2}$-regularity on an open neighborhood of any point  $\bar s\in [0,1]$ where $\dot\gamma(\bar s)\neq 0$ and thus \eqref{weakeq} holds.
\end{proof}
\begin{proof}[\bf  Proof of Theorem \ref{main}]
The implication ``$D$ convex $\Rightarrow$ $\partial D$ convex''
is trivial by using infinitesimal convexity. In fact, if $\partial D$ is not infinitesimally convex
then the normal curvature at some point $x\in\partial D$ is  negative and the
corresponding geodesic yields an immediate contradiction.

For the converse, reasoning as in Remark \ref{bound}, we can consider a family
$(\gam_\eps)_{\eps \in ]0,1]}$ in $\Om$ such that, for any
$\eps\in ]0,1]$, $\gam_\eps$ is a minimum point of $J_\eps$ and a
constant $k>0$ such that, (\ref{up}) holds. Then, by Proposition
\ref{ben} a subsequence $(\gam_{\eps_m})$ exists converging to a
curve $\gamma\in\Omega(p,q;M)$ which satisfies (\ref{weakeq}) for
a.e.  $s\in V$, where $V$ is an open neighborhood of any point $s_0\in A_{\gamma}=\{s\in [0,1]\mid\dot\gamma(s)\neq 0\}$. Let $s_0 \in A_{\gamma}$ be such that $\gam(s_0) \in
D$; then, as $(\gamma_{\eps_m})_m$ uniformly converges to $\gam$,
$\nu\in \N$ and $\delta>0$ exist such that
$$
d= \inf \{\phi(\gam_{\eps_m}(s))\mid s \in [s_0-\delta,s_0+\delta], m \geq \nu\} > 0.
$$
Then $(\lambda_{\eps_m})_m$ uniformly converges to $0$ on $[s_0-\delta,s_0+\delta]$, where it must  be $\lambda(s)=0$ for a.e. $s$
(recall (2) of Proposition~\ref{ben}).
Let now $I$ be a measurable subset in $[0,1]$ with strictly positive measure and assume that $\gam(s)\in \partial D$ for any $s\in I$. Set $I^\ast=I\cap A_{\gamma}$.
Each $s\in  I^\ast$ is a minimum point of $\rho (s)= \phi(\gam(s))$ hence, from \eqref{weakeq}, for a.e.  $s\in I^\ast$ we have
\bmln
0\leq \ddot\rho(s) =
\frac{\partial^2\phi}{\partial x^i\partial x^j}(\gamma(s))\dot\gamma^i(s)\dot\gamma^j(s)
-\frac{\partial\phi}{\partial x^i}(\gam(s))\Gamma^i_{jk}(\gamma(s),\dot\gamma(s))\dot\gamma^j(s)\dot\gamma^k(s)\\
- \lambda(s)\frac{\partial\phi}{\partial x^i}(\gam(s))\frac{\partial\phi}{\partial x^k}(\gam(s))g^{ik}(\gam(s),\dot\gam(s)).\emln
Since $\dot\gamma(s)\in T_{\gamma(s)}\partial D$ for all $s\in I$, from (\ref{fc}), the fact that $0$ is a regular value for $\phi$ and that the matrix $[g^{ik}(x,y)]$ is positive definite, we have  $\lambda(s)=0$, for a.e. $s\in  I^\ast$. Summing up, we have proved that $\lambda(s)=0$, for a.e. $s\in A_{\gamma}$.
By (\ref{weakeq}), this means that $\gam$  is a geodesic  on each connected component of $A_\gamma$ and the function $E(\gamma)=G(\gamma(s),\dot\gamma(s))$ is constant on such connected components. Since $\gamma$ is a $C^1$ curve on $[0,1]$, such constants must agree on the whole interval $[0,1]$, therefore $A_{\gamma}=[0,1]$ and  $\gamma$ is a geodesic  joining $p,q\in D$. Finally, as the boundary is convex, the range of $\gamma$ is contained in $D$.

Moreover $D$ is convex. Indeed,  since $J$ is a continuous functional, recalling that $\gam_{\eps_m}$ is a minimum for $J_{\eps_m}$ and $(\gam_{\eps_m})_m$ converges to $\gam$ in $\Om$
(and therefore $\inf\{\phi(\gamma_{\eps_m}(s))\ |\ s\in[0,1],m\in\N\}>0$), we get
\[J(\gam)=\lim_m J(\gam_{\eps_m})\leq \lim_m J_{\eps_m}(\gam_{\eps_m})\leq \lim_m J_{\eps_m}(\bar \gam)=J(\bar\gam),\]
for any other curve $\bar\gam\in\Om$. Hence $\gamma$ is a minimum for $J$ and therefore also for the length functional ${\ell}_F$.

Now we pass to prove multiplicity of geodesics connecting the points $p$ and $q$ and having support contained in $D$, under the assumption that $D$ is not contractible. This is a quite standard application of Lusternik-Schnirelman theory and its proof is the same as in the case of a domain in a Riemannian manifold.  We observe also that such
geodesics necessarily have different supports,
except if a closed geodesic crosses the given points. We sketch the proof for the reader convenience. We recall that given a topological space $X$ the  Lusternik-Schnirelman category
 of $A\subset X$, denoted by $\cat_X(A)$,  is defined as  the minimum number of closed contractible subsets
of $X$ needed to cover $A$. By definition  $\cat_X(A)=+\infty$ if the covering cannot be realized by a finite number of subsets.
We introduce an auxiliary complete  Riemannian metric $h$ on $M$ and, using a suitable  deformation of the flow of the vector field $\frac{\nabla^h\phi}{1+|\nabla^h\phi|_h^2}$, we can construct as in \cite[Proposition 4.4.8]{mas} a $C^{1,1}$ diffeomorphism $\psi$ of $D$ onto the subset $D\setminus D_{\delta},$ with $D_\delta= \left\{x \in
D\mid \phi(x) < \delta\right\}$. We can use $\psi$
to define, by composition, a locally Lipschitz map on $\Om$ that maps
any sublevel $J^c=\{\gam\in\Om\ |\ J(\gam)\leq c\}$, $c>0$, into the intersection of another sublevel $J^{c'}$, $c'>0$,  with the set of the curves in \Om having support in $D\setminus D_{\delta}$. This is enough to get $\cat_{\Om}(J^c)<+\infty$.
By a result of E. Fadell and A. Husseini \cite{fh}, if $D$ is not contractible,  a sequence $(K_m)_m$ of compact subsets of $\Om$ exists such that, for each $m\in\N$,  $\cat_{\Om}(K_m)\geq m$; hence
for such $m$ and for every $\eps>0$
\[c_{\eps, m}=\inf_{A\in \Gamma_m}\sup_{\gam\in A} J_\eps(\gam),\]
where $\Gamma_m=\{A\subset\Om\ |\ \cat_{\Om}(A)\geq m\}$, is a real number.
Thus $c_{\eps, m}$ is a critical value of functional $J_\eps$ (see for example \cite{pal}).
Observe that for a fixed $c>0$, there must exist $m(c)\in\N$ such that, for any $A\in \Gamma_{m(c)}$, $A\cap (\Om\setminus J^c)\neq\emptyset$, otherwise $\cat_{\Om}(J^c)=+\infty$ (we recall that if $A\subset B$, then $\cat_X(A)\leq \cat_X(B)$).
Therefore, for each $\eps>0$, we have $c\leq c_{\eps,m(c)}\leq\sup_{\gam\in K_{m(c)}} J(\gam)$ and passing to the limit on $\eps\to 0$, by Proposition~\ref{ben} and the first part of this proof, we get a critical value $c_{m(c)}\geq c>0$ of $J$ and therefore a geodesic $\gam_m$ in $D$ connecting $p$ to $q$. Since $c$ was arbitrarily chosen, we obtain in this way a sequence  $(\gam_m)_m\subset \Om$ of geodesics such that,  as $m\to +\infty$, $J(\gam_m)\to +\infty$ and hence  $\ell_F(\gam_m)\to +\infty$ as well.
\end{proof}
\section*{Acknowledgment}
The authors would like to thank Lorenzo D'Ambrosio fore some fruitful conversations about inequality \eqref{avg}.

\end{document}